\documentclass[a4paper,10pt]{amsart}

\usepackage[english]{babel}
\usepackage[utf8]{inputenc}
\usepackage[T1]{fontenc}

\usepackage{amssymb}
\usepackage{amsmath}
\usepackage{amsthm}
\usepackage{bbm}
\usepackage{mathtools}

\usepackage{enumerate}
\usepackage{tikz} 
\usepackage{tikz-cd} 
\usepackage{marginnote}
\usepackage{enumitem}
\usepackage{orcidlink}

\usepackage{csquotes}
\usepackage[style=numeric,isbn=false,url=true,maxnames=50, doi=true, backend = biber]{biblatex}
\AtEveryBibitem{%
  \ifentrytype{thesis}{}{\clearfield{doi}}
  \ifentrytype{book}{\clearfield{url}}{}
  \ifentrytype{article}{\clearfield{url}}{}
  \clearname{translator}%
  \clearname{editorb}%
  \clearfield{pagetotal}%
  \clearfield{editorbtype}%
  \clearfield{month}%
  \clearfield{day}%
  \clearlist{notes}%
  \clearlist{language}%
}%

\DefineBibliographyStrings{english}{%
  edition = {edition},%
}%

\DefineBibliographyStrings{ngerman}{%
  edition = {Edition},%
}%

\newcommand{\bbC}{\mathbb{C}}

\newcommand{\bbE}{\mathbb{E}}

\newcommand{\bbK}{\mathbb{K}}

\newcommand{\bbN}{\mathbb{N}}

\newcommand{\bbR}{\mathbb{R}}

\newcommand{\bbT}{\mathbb{T}}

\newcommand{\bbZ}{\mathbb{Z}}

\newcommand{\calB}{\mathcal{B}}
\newcommand{\calC}{\mathcal{C}}

\newcommand{\calF}{\mathcal{F}}

\newcommand{\calP}{\mathcal{P}}

\DeclareMathOperator{\id}{id}
\DeclareMathOperator{\one}{\mathbbm{1}}

\newcommand{\argument}{\mathord{\,\cdot\,}}

\DeclareMathOperator{\linSpan}{span}

\newcommand{\restricted}[1]{|_{#1}}

\theoremstyle{definition}
\newtheorem{definition}{Definition}[section]
\newtheorem{remark}[definition]{Remark}

\newtheorem{example}[definition]{Example}

\theoremstyle{plain}
\newtheorem{proposition}[definition]{Proposition}
\newtheorem{lemma}[definition]{Lemma}
\newtheorem{theorem}[definition]{Theorem}
\newtheorem{corollary}[definition]{Corollary}

\numberwithin{equation}{section}

\allowdisplaybreaks

\begin{document}

\title[Lower dimension bounds for DMD]{Entropy based lower dimension bounds for finite-time prediction of Dynamic Mode Decomposition algorithms}

\author[T.\ Hauser]{Till Hauser
    \orcidlink{0000-0002-2580-3673}}
\address[T.\ Hauser]{
    Pontificia Universidad Católica de Chile\\
    Facultad de Matemáticas\\
    Edificio Rolando Chuaqui, Campus San Joaquín\\
    Avda.\ Vicuña Mackenna 4860, Macul, Chile}
\email{hauser.math@mail.de}

\author[J.\ Hölz]{Julian Hölz
    \orcidlink{0000-0001-5058-9210}}
\address[J.\ Hölz]{University of Wuppertal, Gaussstr. 20, 42119 Wuppertal, Germany}
\email{hoelz@uni-wuppertal.de}

\date{\today}
\begin{abstract}
    Motivated by Dynamic Mode Decomposition algorithms, we provide lower bounds on the dimension of a finite-dimensional subspace $F \subseteq \mathrm{L}^2(\mathrm{X})$ required for predicting the behavior of dynamical systems over long time horizons. We distinguish between two cases: (i) If $F$ is determined by a finite partition of $X$ we derive a lower bound that depends on the dynamical measure-theoretic entropy of the partition. (ii) We consider general finite-dimensional subspaces $F$ and establish a lower bound for the dimension of $F$ that is contingent on the spectral structure of the Koopman operator of the system, via the approximation entropy of $F$ as studied by Voiculescu. Furthermore, we motivate the use of delay observables to improve the predictive qualities of Dynamic Mode Decomposition algorithms.
\end{abstract}

\maketitle

\section{Introduction}

In recent years, a family of data-driven algorithms known as ``Dynamic Mode Decompositions'' (DMD) has risen to popularity.
We refer to~\cite{Colbrook-MultiverseDMDAlgorithms-2024} for a survey of different DMD algorithms.
One particular member of the family of algorithms is the extended Dynamic Mode Decomposition (eDMD), which allows for partial observation of the system by non-linear observables. Recent developments~\cite{KordaMezic-OnConvergenceEDMDToKoopman-2018} have shown that methods of ergodic theory can be applied to study the limiting behavior of the algorithm in the time-series limit. It is further shown in~\cite{KordaMezic-OnConvergenceEDMDToKoopman-2018} that eDMD converges in the subspace limit.
In this article we will use various notions of entropy from ergodic theory to investigate the subspace limit under the additional assumption that our goal is to use the DMD approximation to predict future values of the observables. By doing this we characterize the rate of convergence of the limit in~\cite[Theorem~5]{KordaMezic-OnConvergenceEDMDToKoopman-2018}.

For this we need the following standard notions.
Let $\mathrm{X} = (X, \Sigma, \mu)$ be a probability space.
A $\Sigma$-measurable mapping $\varphi \colon X \to X$ is called \emph{measure-preserving}, whenever
\begin{align*}
    \mu(A) = \mu(\varphi^{-1}(A))
\end{align*}
holds for all $A \in \Sigma$. In this case the pair $(\mathrm{X}, \varphi)$ is called a \emph{measure-preserving dynamical system}. A measure-preserving dynamical system $(\mathrm{X}, \varphi)$ is called \emph{invertible} if there exists a measure preserving mapping $\varphi^{-1} \colon X \to X$ such that $\mu$-almost surely $\varphi^{-1} \circ \varphi = \varphi \circ \varphi^{-1} = \id_X$.

The \emph{Koopman} or \emph{composition operator} (on $\mathrm{L}^2(\mathrm{X})$) of a measure-preserving dynamical system $(\mathrm{X}, \varphi)$ is the linear operator defined by
\begin{align*}
    T_\varphi \colon \mathrm{L}^2(\mathrm{X}) \to \mathrm{L}^2(\mathrm{X}), \quad f \mapsto f \circ \varphi.
\end{align*}
Here we will view $\mathrm{L}^2(\mathrm{X})$ as the space of square integrable functions $f \colon X \to \bbK$ that are identified when they agree on all $\mu$-null sets and endowed with the usual $2$-norm. The field $\bbK$ is either $\bbR$ or $\bbC$, although in Section~\ref{sec:theSpectralPerspective} we will require $\bbK = \bbC$ to do spectral theory.
The Koopman operator can be viewed as a linearization of $\varphi$
that reverses time.

The extended Dynamic Mode Decomposition algorithm aims at approximating $T_\varphi$ on a closed linear subspace $F \subseteq \mathrm{L}^2(\mathrm{X})$, which is usually finite-dimensional. More precisely, it seeks to construct an operator $\hat{T}_\varphi \colon F \to F$ that minimizes the $\mathrm{L^2}$-error
\begin{align}
    \label{eq:objective-DMD}
    {\lVert \hat{T}_\varphi f  - T_\varphi f \rVert}_2
\end{align}
simultaneously for every $f \in F$. Let $P_F \colon \mathrm{L}^2(\mathrm{X}) \to \mathrm{L}^2(\mathrm{X})$ be the orthogonal projection onto $F$. Since for all $f, g \in F$ the inequality
\begin{align*}
    {\lVert P_F T_\varphi f - T_\varphi f \rVert}_2 \leq {\lVert g - T_\varphi f \rVert}_2.
\end{align*}
holds, a precise solution $\hat T_\varphi$ to this optimization problem exists and is given by the operator $\hat T_\varphi \coloneqq P_F T_\varphi \restricted{F}$. It has been shown in~\cite[Theorem~2]{KordaMezic-OnConvergenceEDMDToKoopman-2018} that, under some assumption on a spanning set $\omega$ of $F$ (called the \emph{dictionary}), the extended Dynamical Mode Decomposition algorithm converges to the linear operator $P_F T_\varphi \restricted{F}$ in the time-series limit.

Note that with limited data the eDMD algorithm will not exactly return $\hat T_\varphi$ exactly but only an approximation of $\hat T_\varphi$. We denote this approximation by $\hat T$ and use it as an approximation of the original Koopman operator $T_\varphi$ on the subspace $F$. An algorithm that computes the best approximation $\hat T_\varphi$ exactly is known as the ``analytic Dynamic Mode Decomposition'~\cite[Section~7]{KordaMezic-OnConvergenceEDMDToKoopman-2018}.

Let  $\hat T \colon F \to F$ be any approximation of $T_\varphi$ on the subspace $F$ (not necessarily an approximation from the DMD algorithms). We can use $\hat T$ to approximate trajectories $f, T_\varphi f, \dots, T_\varphi^{K - 1}f$ of the Koopman operator with trajectories of the approximation $f, \hat T f , \dots, \hat T^{K - 1} f$. Ideally we would like the error
\begin{align*}
    {\lVert \hat T^k f - T_\varphi^k f \rVert}_2
\end{align*}
to be small for all times $k$ in a given finite time horizon $\{0, \dots, K - 1\}$.

It is fairly simple to provide a rough error bound for trajectory errors. If the approximation $\hat T$ is assumed to be ``good enough'', meaning that there is an $\varepsilon > 0$ such that $\lVert \hat T f - T_\varphi f \rVert \leq \varepsilon \lVert f \rVert$ for each $f \in F$ and if $\hat T$ is a contraction, we may conclude that
\begin{align*}
    \lVert \hat T^{k} f - T_\varphi^{k} f \rVert = \Bigg\lVert \sum_{j = 0}^{k - 1} T_\varphi^j (\hat T - T_\varphi) \underbrace{\hat T^{(k - 1) - j} f}_{\in F} \Bigg\rVert \leq k \varepsilon \lVert f \rVert
\end{align*}
for all $f \in F$ and all $k \in \bbN$. For small times this estimate provides bounds on the prediction error, however, for large $k$ this estimate does not provide any value as the bound $\lVert \hat T^{k} f - T_\varphi^{k} f \rVert \leq 2 \lVert f \rVert$ is trivial for every $f \in F$.

Instead of finding a bound that grows linearly in $k$, we may ask for conditions on a subspace $F \subseteq \mathrm{L}^2(\mathrm{X})$ such that for some $\varepsilon > 0$ and some time horizon $K \in \bbN$ we have that
\begin{align}
    \label{eq:error-bound-finite-time-horizon}
    \lVert \hat T^{k} f - T_\varphi^{k} f \rVert \leq \varepsilon \lVert f \rVert
\end{align}
holds for all $f \in F$ and all $k \in \{0, \dots, K - 1\}$.
In case $\hat T = \hat T_\varphi$ it has been shown in~\cite[Theorem~5]{KordaMezic-OnConvergenceEDMDToKoopman-2018} that such a subspace exists and that, if given an orthonormal base of $(f_n)_{n \in \bbN}$ of $\mathrm{L}^2(\mathrm{X})$, then the subspace $F$ can be chosen to be the linear span of $f_1, \dots, f_N$ for some large enough $N \in \bbN$.

In numerical analysis, a natural question arises: What is the minimal dimension of a linear subspace $F$ that provides the desired error bound~\eqref{eq:error-bound-finite-time-horizon}. In this article, we establish lower bounds on the dimension of $F$ required for accurately predicting the behavior of dynamical systems over long time horizons. In this endeavor, we distinguish between two cases:

\subsection{Lower bounds by measure-theoretic entropy}

Consider a measure-pre\-{}serving system $(\mathrm{X},\varphi)$.
Measure-theoretic entropy $h_\mu(\varphi, \alpha)$ of a finite $\Sigma$-measurable partition $\alpha$ with respect to $\varphi$ was introduced by Kolmogorov and Sinai and has since then become a cornerstone of ergodic theory. We will give a precise definition of this concept in Section~\ref{sec:finite-time-prediction-measure-entropy} and recommend~\cite{Walters-IntroductionToErgodicTheory-1982} for an excellent exposition.
Roughly speaking the measure-theoretical entropy of $\alpha$ is a quantification of the unpredictibility present in a measure-preserving system observing it at the resolution $\alpha$.
It is thus natural to ask, whether the presence of measure-theoretical entropy for a finite $\Sigma$-measurable partition $\alpha$ limits the possibilities of long term predictions via DMD algorithms.
For this we assume that the subspace $F$ is of the form
\begin{align*}
    F = \mathrm{L}^2(\mathrm{X} \mid \alpha)  \coloneqq \linSpan \{ \one_A : A \in \alpha \}
\end{align*}
for some finite $\Sigma$-measurable partition $\alpha$ of $X$, i.e., we choose the dictionary to consist only of functions that tell us whether the state of the system is currently in some set $A \in \alpha$. In this case, we can derive a lower bound for the dimension of $F$, which depends on the measure-theoretic entropy $h_\mu(\varphi, \alpha)$ of the partition $\alpha$ with respect to $\varphi$.
Note that this choice of a subspace $F$ is used in the multiplicative Dynamic Mode Decomposition algorithm (see~\cite[Section~3]{BoulleColbrook-MultiplicativeDynamicModeDecomposition-2024}) and has been considered in~\cite[Sections~3.2 and~4.1]{SchuetteKoltaiKlus-OnNumericalApproximationPerronFrobeniusKoopman-2016}. See Subsection~\ref{subsec:proofOfTheoremMeasureTheoreticalEntropyLowerBound} for the proof of the following.

\begin{theorem}
    \label{thm:exponential_lower_bound_dimension}
    Let $(\mathrm{X}, \varphi)$ be a measure-preserving system, $\varepsilon > 0$ and $\alpha$ be a finite $\Sigma$-measurable partition of $X$. Then there exists $K_0 \in \bbN$ and $\delta > 0$ that satisfy the following: Whenever for a linear operator $\hat T \colon \mathrm{L}^2(\mathrm{X} \mid \alpha) \to \mathrm{L}^2(\mathrm{X} \mid \alpha)$ and some $K \geq K_0$ the assertion
    \begin{align*}
        \lVert \hat T^k f - T_\varphi^k f \rVert \leq \delta \lVert f \rVert
    \end{align*}
    holds for all $f \in F$ and $k \in \{0, \dots, K - 1\}$, we conclude that
    \begin{align}\label{eq:thm:exponential_lower_bound_dimension:condition}
        \lvert \alpha \rvert \geq \exp(K ( h_\mu(\varphi, \alpha) - \varepsilon)).
    \end{align}
    Here $\delta$ only depends on the system $(\mathrm{X}, \varphi)$, the choice of $\varepsilon$ and the cardinality of $\alpha$.
\end{theorem}

Note that the measure-theoretic entropy $h_\mu(\varphi)$ of a measure-preserving system $(\mathrm{X},\varphi)$ is defined as the supremum over all $h_\mu(\varphi,\alpha)$ ranging over all finite $\Sigma$-measurable partitions $\alpha$. Thus, from the theorem we observe that whenever the measure-preserving system $(\mathrm{X},\varphi)$ has non-zero measure-theoretic entropy, then care has to be taken while choosing $\alpha$. Note that $h_\mu(\varphi, \alpha) = 0$ if and only if $\alpha$ is measurable with respect to the Pinsker-$\sigma$-algebra of $\varphi$ (see~\cite[Section~18.1]{Glasner-ErgodicTheoryViaJoinings-2003}).

\subsection{Lower bounds by approximation entropy}

Recall that Theorem~\ref{thm:exponential_lower_bound_dimension} applies to finite-dimensional subspaces of the form $\mathrm{L}^2(\mathrm{X} \mid \alpha)$ for a finite $\Sigma$-measurable partition $\alpha$ of $X$.
It is natural to ask, whether it can be generalized to arbitrary finite-dimensional subspaces $F$ of $\mathrm{L}^2(\mathrm{X})$.
For this we consider the \emph{approximation entropy $h_{\mathrm{apr}}(T_\varphi,F)$ of $T_\varphi$ with respect to $F$} introduced by Voiculescu in~\cite[Section~7]{Voiculescu-DynamicalApproximationEntropies-1995}, which gives another quantification of the unpredictability of a measure-preserving system. For details on this notion and the proof of the following see Section~\ref{sec:lower-bounds-multiplicity-function}.

\begin{theorem}
    \label{thm:approximation-entropy-Koopman}
    Let $(\mathrm{X}, \varphi)$ be an invertible measure-preserving system, let $F \subseteq \mathrm{L}^2(\mathrm{X})$ be a finite-dimensional linear subspace and $\varepsilon > 0$. Then there exists $K_0 \in \bbN$ and $\delta > 0$ that satisfy the following: Whenever for a linear operator $\hat T \colon F \to F$  and some $K \geq K_0$ the assertion
    \begin{align*}
        \lVert \hat T^k f - T_\varphi^k f \rVert \leq \delta \lVert f \rVert
    \end{align*}
    holds for all $x \in F$ and all $k \in \{0, \dots, K - 1\}$, we conclude that
    \begin{align*}
        \dim(F) \geq K (h_{\mathrm{apr}}(T_\varphi, F) - \varepsilon).
    \end{align*}
\end{theorem}

We will present in Theorem~\ref{thm:hApproxInfinity} that whenever $\varphi$ has non-zero measure-theoretic entropy (and in Proposition~\ref{prop:contLebesgucomp-implies-approxEntInfty} that if $T_\varphi$ has a countable Lebesgue component), then there exist finite-dimensional subspaces $F$ of $\mathrm{L}^2(\mathrm{X})$, which allow for arbitrary large values of $h_{\mathrm{apr}}(T_\varphi, F)$. Thus, as above, care has to be taken by choosing $F$ in this context.

Let $F \subseteq \mathrm{L}^2(\mathrm{X})$ be a finite-dimensional subspace and $n \in \bbN$. The \emph{$n$-th delay subspace} for $(\mathrm{X}, \varphi)$ is given by
\begin{align*}
    F_n \coloneqq \linSpan \bigcup_{k = 0}^n T_\varphi F.
\end{align*}
Note that its dimension is upper bounded by $n \cdot \dim(F)$.
We will see in Section~\ref{sec:lower-bounds-multiplicity-function} that we have $h_{\mathrm{apr}}(T_\varphi, F_n)=h_{\mathrm{apr}}(T_\varphi, F)$. It is thus natural to ask for a stronger statement of Theorem~\ref{thm:approximation-entropy-Koopman} by replacing $F$ with $F_n$. This is possible as we will see in Section~\ref{sec:lower-bounds-multiplicity-function}. The following theorem is a special case of Theorem~\ref{thm:linear_lower_bound_dimension}.

\begin{theorem}
    \label{thm:delay-observables-Koopman}
    Let $(\mathrm{X}, \varphi)$ be an invertible measure-preserving system, let $F \subseteq \mathrm{L}^2(\mathrm{X})$ be a finite-dimensional linear subspace and $\varepsilon > 0$. Then there exists $K_0 \in \bbN$ and $\delta > 0$ that satisfy the following: Whenever for some $n \in \bbN$, a linear operator $\hat T_n \colon F_n \to F_n$ and $K \geq K_0$ the assertion
    \begin{align*}
        \lVert \hat T^k_n f - T_\varphi^k f \rVert \leq \delta \lVert f \rVert
    \end{align*}
    holds for all $x \in F_n$ and all $k \in \{0, \dots, K - 1\}$, we conclude that
    \begin{align*}
        \dim(F_n) \geq K (h_{\mathrm{apr}}(T_\varphi, F) - \varepsilon).
    \end{align*}
\end{theorem}

\subsection{Structure of the article}

After a short introduction to measure-theoretic entropy $h_\mu$ and conditional expectations in the Subsections~\ref{subsec:Measure-theoretic-entropy} and~\ref{subsec:conditionalExpectations} we provide the proof of Theorem~\ref{thm:approximation-entropy-Koopman} in Subsection~\ref{subsec:proofOfTheoremMeasureTheoreticalEntropyLowerBound}.
In Section~\ref{sec:lower-bounds-multiplicity-function} we define the approximation entropy $h_{\mathrm{apr}}$ and present a Hilbert space version of Theorem~\ref{thm:delay-observables-Koopman}, from which Theorem~\ref{thm:approximation-entropy-Koopman} and Theorem~\ref{thm:delay-observables-Koopman} can easily be deduced.
We continue the discussion of the approximation entropy in Section~\ref{sec:interplayMeasureTheoreticalAndApproximationEntropy} by relating it to the measure-theoretic entropy via the notion of countable Lebesgue components.
In Section~\ref{sec:theSpectralPerspective} we present some details on the relationship between the approximation entropy and the spectral theory and provide references for further reading.

\section{Lower bounds by measure-theoretic entropy}
\label{sec:finite-time-prediction-measure-entropy}

\subsection{Measure-theoretic entropy}
\label{subsec:Measure-theoretic-entropy}

See~\cite{Walters-IntroductionToErgodicTheory-1982} for an excellent exposition of measure-theoretical entropy and its intuition.
Let $\mathrm{X} = (X, \Sigma, \mu)$ be a probability space.
A partition $\beta$ of $X$ into $\Sigma$-measurable subsets is called a \emph{$\Sigma$-measurable} partition.
In the following we use the convention that sums over $\beta$ avoid sets with $\mu$-measure zero.
We define the \emph{static entropy of $\alpha$} by
\begin{align*}
    H_\mu(\alpha)\coloneqq \sum_{A \in \alpha} - \mu(A) \log(\mu(A)).
\end{align*}
Note that the static entropy of $\alpha$ can be intuitively understood as a quantification of the information provided by thinking of $\alpha$ as a resolution.
For two finite $\Sigma$-measurable partitions $\alpha$ and $\beta$ of $X$, we set
\begin{align*}
    \alpha \vee \beta \coloneqq \{ A \cap B : A \in \alpha, \, B \in \beta \}
\end{align*}
to be the \emph{common refinement} of $\alpha$ and $\beta$.
Note that the common refinement is associative and hence allows to define the common refinement $\bigvee_{k=1}^n \alpha_n$ also for finite families $(\alpha_k)_{k=1}^n$ of finite $\Sigma$-measurable partitions.
We define the \emph{conditional static entropy of $\alpha$ with respect to $\beta$} by $H_\mu(\alpha \mid \beta) \coloneqq H_\mu(\alpha \vee \beta) - H_\mu(\beta)$. The intuition behind this notion is to quantify the new information gained by learning $\alpha$, after having known $\beta$ already.

Now consider a measure-preserving dynamical system $(\mathrm{X}, \varphi)$.
For a finite \linebreak $\Sigma$-measurable partition $\alpha$ we denote $\varphi^{-k}(\alpha):=\{\varphi^{-k}(A);\, A\in \alpha\}$ for $k\in \mathbb{N}$.
We define the \emph{measure-theoretic entropy of $\alpha$ with respect to $\varphi$} by
\begin{align*}
    h_\mu(\varphi, \alpha) \coloneqq \lim_{n \to \infty} \frac{1}{n} H_\mu \left( \bigvee_{k = 0}^{n-1} \varphi^{-k}(\alpha) \right).
\end{align*}
Note that the existence of the limit can be shown with a standard subadditivity argument involving Fekete's lemma~\cite[Corollary 4.9.1]{Walters-IntroductionToErgodicTheory-1982}.
Taking a supremum over all finite $\Sigma$-measurable partitions $\alpha$ of $X$ we define the \emph{measure-theoretic entropy of $\varphi$} by
\begin{align*}
    h_\mu(\varphi) \coloneqq \sup_\alpha h_\mu(\varphi, \alpha).
\end{align*}

\subsection{Projections and conditional expectations}
\label{subsec:conditionalExpectations}

A linear operator $T \colon \mathrm{L}^2(\mathrm{X}) \to \mathrm{L}^2(\mathrm{X})$  is called \emph{Markov operator} if $Tf \geq 0$ for all $0 \leq f \in \mathrm{L}^2(\mathrm{X})$ and $T \one = \one$. A projection $P: \mathrm{L}^2(\mathrm{X}) \to \mathrm{L}^2(\mathrm{X})$ that is a Markov operator is called \emph{Markov projection}. For further details on Markov operators see~\cite[Chapter~13]{EisnerFarkasHaaseNagel-OperatorTheoreticAspectsErgodicTheory-2015}.

Consider a finite $\Sigma$-measurable partition $\alpha$ of a probability space $\mathrm{X} = (X,\Sigma,\mu)$.
We denote
\begin{align*}
    \mathrm{L}^2(\mathrm{X} \mid \alpha) \coloneqq \linSpan \{ \one_A : A \in \alpha \}.
\end{align*}
Note that this space consists precisely of the measurable functions with respect to the finite $\sigma$-algebra generated by $\alpha$. Furthermore, it forms a unital sublattice of $\mathrm{L}^2(\mathrm{X})$. For details see~\cite[Proposition~III.11.2]{Schaefer-BanachLatticesPositiveOperators-1974} or~\cite[Proposition~13.19~b)]{EisnerFarkasHaaseNagel-OperatorTheoreticAspectsErgodicTheory-2015}.
A linear operator $P \colon \mathrm{L}^2(\mathrm{X}) \to  \mathrm{L}^2(\mathrm{X})$ is called the \emph{conditional expectation} of $\alpha$, whenever it satisfies
\begin{align*}
    \int_A Pf \, \mathrm{d} \mu = \int_A f \, \mathrm{d} \mu
\end{align*}
for all $A \in \alpha$ and all $f \in \mathrm{L}^2(\mathrm{X})$.
This operator is uniquely defined and we denote $\bbE[f \mid \alpha] \coloneqq Pf$. Note that for $f \in \mathrm{L}^2(\mathrm{X})$
\begin{align*}
    \bbE[f \mid \alpha] = \sum_{A \in \alpha} \frac{1}{\mu(A)} \int_{A} f \, \mathrm{d} \mu \cdot \one_{A}.
\end{align*}
For further details on conditional expectations we recommend~\cite[Section~34]{Billingsley-ProbabilityAndMeasure-1995}.
For the proof of the following see~\cite[Lemma~13.16, Remark~13.21]{EisnerFarkasHaaseNagel-OperatorTheoreticAspectsErgodicTheory-2015}.

\begin{remark}
    \label{rem:projections-equivalence}
    Let $\alpha$ be a $\Sigma$-measurable partition of $X$ and let $P \colon \mathrm{L}^2(\mathrm{X}) \to \mathrm{L}^2(\mathrm{X})$ be a bounded linear operator. Then the following statements are equivalent.
    \begin{enumerate}[label = (\roman*)]
        \item\label{itm:rem:projections-equivalence:orthogonal}
        $P$ is the orthogonal projection onto $\mathrm{L}^2(\mathrm{X} \mid \alpha)$.
        \item\label{itm:rem:projections-equivalence:Markov}
        $P$ is a \emph{Markov projection} onto $\mathrm{L}^2(\mathrm{X} \mid \alpha)$.

        \item\label{itm:rem:projections-equivalence:conditional}
        $P$ is the \emph{conditional expectation} $\bbE[\argument \mid \alpha]$.
    \end{enumerate}
\end{remark}

\begin{remark}[Markov structure]
    Recall from the introduction that the optimal solution to the objective in~\eqref{eq:objective-DMD} is given by $\hat T_\varphi = P_F T_\varphi \restricted{F}$, where $T_\varphi$ denotes the Koopman operator on $\mathrm{L}^2(\mathrm{X})$ and $P$ is the orthogonal projection onto $\mathrm{L}^2(\mathrm{X} \mid \alpha)$. Since $P$ and $T_\varphi$ are both Markov operators their composition $\hat T_\varphi$ is as well.
\end{remark}

\subsection{Proof of Theorem~\ref{thm:exponential_lower_bound_dimension}}
\label{subsec:proofOfTheoremMeasureTheoreticalEntropyLowerBound}

In this section we will provide the proof of Theorem~\ref{thm:exponential_lower_bound_dimension} and prepare some arguments for later use.

\begin{lemma}
    \label{lem:partitionInequalityI}
    For any finite $\Sigma$-measurable partition $\beta$ of $X$ and any measurable $A \subseteq X$ we have
    \begin{align*}
        {\lVert \bbE [ \one_A \mid \beta ]- \one_A \rVert}_1
         & = 2 \sum_{B \in \beta}  \frac{\mu(A\cap B) \mu(B\setminus A)}{\mu(B)} .
    \end{align*}
\end{lemma}
\begin{proof}
    We compute
    \begin{align*}
        \bbE [ \one_A \mid \beta ]- \one_A
         & =  \sum_{B \in \beta} \frac{\mu(A \cap B)}{\mu(B)}  \one_{B} - \sum_{B \in \beta} \one_{A} \cdot \one_B                                            \\
         & = \sum_{B \in \beta} \left( \left(\frac{\mu(A \cap B)}{\mu(B)}-1\right)  \one_{A\cap B} +  \frac{\mu(A\cap B)}{\mu(B)} \one_{B\setminus A}\right).
    \end{align*}
    Thus,
    \begin{align*}
        \left\lvert \bbE [ \one_A \mid \beta ]- \one_A \right\rvert
         & = \sum_{B \in \beta} \left( \left( 1 - \frac{\mu(A \cap B)}{\mu(B)}\right)  \one_{A\cap B} +  \frac{\mu(A\cap B)}{\mu(B)} \one_{B\setminus A}\right) \\
         & = \sum_{B \in \beta} \left(\frac{\mu(B\setminus A)}{\mu(B)} \one_{A\cap B} +  \frac{\mu(A\cap B)}{\mu(B)} \one_{B\setminus A}\right)
    \end{align*}
    and we observe that
    \begin{align*}
        {\lVert \bbE [ \one_A \mid \beta ]- \one_A \rVert}_1
         & = 2 \sum_{B \in \beta}  \frac{\mu(A\cap B) \mu(B\setminus A)}{\mu(B)} .
    \end{align*}
\end{proof}

For the following lemma we adapt and extend an argument given in~\cite[Lemma 4.15]{Walters-IntroductionToErgodicTheory-1982} to our context.

\begin{lemma}
    \label{lem:continuity-conditional-expectation}
    Let $\kappa\in \mathbb{N}$,
    $p \in [1, \infty)$ and
    $\varepsilon > 0$.
    Then there exists $\delta > 0$, such that
    for any finite $\Sigma$-measurable partitions $\alpha$ and $\beta$ of $X$ with $|\alpha|\leq \kappa$ and $\sup_{A\in \alpha} {\lVert \bbE[\one_A \mid \beta ] - \one_A \rVert}_p < \delta$, we obtain that $H_\mu(\alpha \mid \beta) < \varepsilon$.
\end{lemma}
\begin{proof}
    Let $c \coloneqq {\lVert \one \rVert}_q$, where $q \in (1, \infty]$ is the Hölder conjugate of $p$. Choose $\delta \in \left(0, \frac{1}{4 c \kappa} \right)$ such that $- \kappa c \delta \log( c \delta ) - (1 - \kappa c \delta) \log(1 - \kappa c \delta) < \varepsilon$.

    Now consider finite $\Sigma$-measurable partitions $\alpha$ and $\beta$ of $X$ with $|\alpha|\leq \kappa$ and $\sup_{A\in \alpha} {\lVert \bbE[\one_A \mid \beta ] - \one_A \rVert}_p < \delta$. For $A \in \alpha$ we denote by $\beta_A$ the collection of sets $B \in \beta$ for which we have
    \begin{align*}
        \frac{\mu(B \cap A)}{\mu(B)} > \frac{1}{2}.
    \end{align*}
    Then $\beta \setminus \beta_A$ contains all the sets $B$ such that
    \begin{align*}
        \frac{\mu(B \setminus A)}{\mu(B)} = 1 - \frac{\mu(B \cap A)}{\mu(B)} \geq \frac{1}{2}.
    \end{align*}
    We thus observe for
    $C_A \coloneqq \bigcup_{B \in \beta \setminus \beta_A} (A \cap B)$ that
    \begin{align*}
        \mu \left(C_A\right)
         & \leq \sum_{B \in \beta_A} \mu(B \setminus A) + \sum_{B \in \beta \setminus \beta_A} \mu(A \cap B) \\
         & \leq 2 \sum_{B \in \beta} \frac{\mu(B \setminus A) \, \mu(A \cap B)}{\mu(B)}.
    \end{align*}
    Thus Lemma~\ref{lem:partitionInequalityI} and the Hölder inequality imply
    \begin{align}
        \label{eq:Cestimate}
        \mu \left(C_A \right)
         & \leq {\lVert \bbE [ \one_A \mid \beta ]- \one_A \rVert}_1
        \leq c{\lVert \bbE [ \one_A \mid \beta ]- \one_A \rVert}_p
        <c\delta
    \end{align}
    for all $A\in \alpha$.
    Define the index set $\Lambda \coloneqq  \{ (A,B) : B \in \beta_A, \, A \in \alpha \}$ and set $C \coloneqq \bigcup_{(A,B) \in \Lambda} A\cap B$. Then it is straightforward to observe that $\gamma \coloneqq \{C_A : A\in \alpha \} \cup \{C \}$ is a finite $\Sigma$-measurable partition of $X$.
    Notice that $\alpha \vee \beta \subseteq \gamma \vee \beta$. Indeed, if $A \in \alpha$ and $B \in \beta \setminus \beta_A$, then
    \begin{align*}
        A \cap B = \left( \bigcup_{\tilde B \in \beta \setminus \beta_A} (A \cap \tilde B) \right) \cap B\in \gamma \vee \beta.
    \end{align*}
    And if $A \in \alpha$ and $B \in \beta_A$, then
    \begin{align}
        \label{eq:helper-common-refinement-inclusion}
        \gamma \vee \beta \ni \left( \bigcup_{(\tilde A, \tilde B) \in \Lambda} (\tilde A \cap \tilde B) \right) \cap B = \bigcup \{ \tilde A \cap B : \tilde A \in \alpha \text{ such that } B \in \beta_{\tilde A} \}.
    \end{align}
    Notice that $B$ is only an element of $\beta_{\tilde A}$ for at most one $\tilde A \in \alpha$. Otherwise, if there are two $A_1, A_2 \in \alpha$ such that $B \in \beta_{A_1} \cap \beta_{A_2}$, then by disjointness of $\tilde A_1$ and $\tilde A_2$ we obtain
    \begin{align*}
        \frac{\mu(B \cap (\tilde A_1 \cup \tilde A_2)) }{\mu(B)} = \frac{\mu(B \cap \tilde A_1) }{\mu(B)} + \frac{\mu(B \cap \tilde A_2) }{\mu(B)} > \frac{1}{2} + \frac{1}{2} > 1.
    \end{align*}
    This is a contradiction. Continuing~\eqref{eq:helper-common-refinement-inclusion}, it then follows that
    \begin{align*}
        \gamma \vee \beta \ni \bigcup \{ \tilde A \cap B : \tilde A \in \alpha \text{ such that } B \in \beta_{\tilde A} \} = A \cap B.
    \end{align*}
    Hence, the inclusion $\alpha \vee \beta \subseteq \gamma \vee \beta$ holds.

    Now recall from~\eqref{eq:Cestimate} that $\mu(C_A) < c \delta$ for all $A \in \alpha$. This implies that $\mu(C) \geq 1 - \lvert \alpha \rvert c \delta \geq 1 - \kappa c \delta$ and we observe
    \begin{align*}
        H_\mu(\gamma)
         & \leq -\lvert \alpha \rvert c \delta \log(c \delta)-(1 - \kappa c \delta) \log(1 - \kappa c \delta) \\
         & \leq -\kappa c \delta \log(c \delta) - (1 - \kappa c \delta) \log(1 - \kappa c \delta)             \\
         & < \varepsilon.
    \end{align*}
    Thus, since $\alpha \vee \beta \subseteq \gamma \vee \beta$, we obtain from~\cite[Theorem~4.3]{Walters-IntroductionToErgodicTheory-1982} that
    \begin{align*}
        H_\mu(\beta) + H_\mu( \alpha \mid \beta)
        = H_\mu (\alpha \vee \beta) 
        = H_\mu (\gamma \vee \beta) 
        \leq H_\mu(\gamma) + H_\mu(\beta) < H_\mu(\beta) + \varepsilon, 
    \end{align*}
    which implies $H_\mu(\alpha \mid \beta) < \varepsilon$.
\end{proof}

One last simple but important observation is the following.

\begin{lemma}
    \label{lem:control}
    Let $T \colon \mathrm{L}^2(\mathrm{X}) \to \mathrm{L}^2(\mathrm{X})$ be bounded linear operator and denote by $F \subseteq \mathrm{L}^2(\mathrm{X})$ a finite-dimensional linear subspace. Let further $\hat T \colon F \to F$ be linear, $P_F$ be the orthogonal projection onto $F$ and $f \in F$ as well as $k \in \bbN$. If there exists $\delta > 0$ such that  ${\lVert \hat T^k f - T^k f \rVert}_2 \leq \delta {\lVert f \rVert}_2$, then also
    ${\lVert P_F T^k f - T^k f \rVert}_2 \leq \delta {\lVert f \rVert}_2$.
\end{lemma}
\begin{proof}
    Since $P_F$ is an orthogonal projection onto $F$, the inequality
    \begin{align*}
        \lVert P_F T^k f - T^k f \rVert \leq \lVert g - T^k f \rVert
    \end{align*}
    holds for all $f, g \in F$. Hence, setting $g \coloneqq \hat T^k f$ we obtain
    \begin{align*}
        \lVert P_F T^k f - T^k f \rVert \leq \lVert \hat T^k f - T^k f \rVert \leq \delta \lVert f \rVert.
    \end{align*}
    Thus, the statement follows.
\end{proof}

We are now ready to prove the main result of this section, which was stated in the introduction.

\begin{proof}[Proof of Theorem~\ref{thm:exponential_lower_bound_dimension}]
    Let $\varepsilon > 0$ and $\alpha$ be the finite $\Sigma$-measurable partition of $X$. Let $F \coloneqq \mathrm{L}^2(\mathrm{X} \mid \alpha)$ and denote by $P_F$ the orthogonal projection onto $F$. It follows from Remark~\ref{rem:projections-equivalence} that $P_F$ coincides with the conditional expectation $\bbE[\argument \mid \alpha]$. So by Lemma~\ref{lem:continuity-conditional-expectation} there exists $\delta > 0$ such that whenever $\lVert P_F T_\varphi^k \one_A - T_\varphi^k \one_A \rVert < \delta$ for some $k \in \bbN$, we obtain that $H_\mu (\varphi^{-k}(\alpha) \mid \alpha) < \varepsilon/2$. Also notice that there exists $K_0 \in \bbN$ such that
    \begin{align*}
        \left\lvert h_\mu(\varphi, \alpha) - \frac{1}{K} H_\mu \left( \bigvee_{k = 0}^{K - 1} \varphi^{-k}(\alpha) \right) \right\rvert < \varepsilon/2
    \end{align*}
    for all $K \geq K_0$.

    Let $f \in F$ and fix $K \geq K_0$ and assume that ${\lVert \hat T^k f - T_\varphi^k f \rVert}_2 \leq \delta {\lVert f \rVert}_2$ holds for all $f \in F$ and all $k \in \{0, \dots, K - 1\}$. From Lemma~\ref{lem:control} we observe that
    $\lVert P_F T_\varphi^k f - T_\varphi^k f \rVert \leq \delta \lVert f \rVert$ for all $f \in F$ and all $k \in \{0, \dots, K - 1\}$. In particular, this implies that $\lVert P_F T_\varphi^k \one_A - T_\varphi^k \one_A \rVert \leq \delta$ holds for all $k \in \{0, \dots, K_0\}$ and all $A \in \alpha$. Hence, Lemma~\ref{lem:continuity-conditional-expectation} yields that
    \begin{align*}
        H_\mu(\varphi^{-k}(\alpha) \mid \alpha) \leq \varepsilon/2
    \end{align*}
    for all $k \in \{0, \dots, K - 1\}$. Thus,
    \begin{align*}
        H_\mu \left( \bigvee_{k = 0}^{K - 1} \varphi^{-k}(\alpha) \right) & \leq H_\mu \left( \alpha \vee \left( \bigvee_{k = 1}^{K-1} \varphi^{-k}(\alpha) \right) \right)              \\
                                                                          & \leq H_\mu(\alpha) + H_\mu \left( \bigvee_{k = 1}^{K - 1} \varphi^{-k}(\alpha) \, \middle| \, \alpha \right) \\
                                                                          & \leq H_\mu(\alpha) + \sum_{k = 1}^{K - 1} H_\mu (\varphi^{-k} (\alpha) \mid \alpha)                          \\
                                                                          & \leq \log(\lvert \alpha \rvert) + \tfrac{(K - 1) \varepsilon}{2}.
    \end{align*}
    Now dividing by $K$ yields
    \begin{align*}
        h_\mu(\varphi, \alpha) - \varepsilon \leq \frac{1}{K} H_\mu \left(\bigvee_{k = 0}^{K - 1} \varphi^{-k}(\alpha) \right) - \varepsilon/2 \leq \frac{1}{K} \log(\lvert \alpha \rvert),
    \end{align*}
    and the statement follows.
\end{proof}

\section{Lower bounds by approximation entropy}
\label{sec:lower-bounds-multiplicity-function}

In this section we provide a proof for Theorem~\ref{thm:delay-observables-Koopman}. In order to do this we show a more general statement about unitary operators on Hilbert spaces.
Let $H$ be a Hilbert space, denote by $\calF_{\mathrm{fin}}(H)$ the space of all finite-dimensional subspaces $F \subset H$ and by $\calP_{\mathrm{orth}}(H)$ the space of all orthogonal projections $P \colon H \to H$. Consider for $\delta > 0$ the quantity measuring the complexity of approximating a finite point set $\omega \subseteq H$ by a finite-dimensional subspace $F \subseteq H$ given by
\begin{align*}
    H_{\mathrm{apr}}(\omega, \delta)
    \coloneqq & \min \{ \dim F : F \in \calF_{\mathrm{fin}}(H), \, \forall x \in \omega: \exists y \in F : \lVert x - y \rVert < \delta \} \\
    =         & \min \{ \dim(PH): P \in \calP_{\mathrm{orth}}(H), \, \forall x \in \omega : \lVert Px - x \rVert < \delta \}.
\end{align*}
Notice that $H_{\mathrm{apr}}(\omega, \delta)$ is bounded from above by $\lvert \omega \rvert$.
Considering a unitary operator $T\colon H\to H$, a standard subadditivity argument yields the existence of the following limit. See for details~\cite[Proposition~7.13]{Voiculescu-DynamicalApproximationEntropies-1995}. We denote
\begin{align*}
    h_{\mathrm{apr}}(T, \omega, \delta) & \coloneqq \lim_{n \to \infty} \frac{1}{n} H_{\mathrm{apr}} \left( \bigcup_{k = 0}^{n - 1} T^k \omega, \delta \right)
\end{align*}
and define the \emph{approximation entropy of $\omega$ with respect to $T$} by
\begin{align*}
    h_{\mathrm{apr}}(T, \omega)\coloneqq\sup_{\delta>0}h_{\mathrm{apr}}(T, \omega, \delta).
\end{align*}
Note that $h_{\mathrm{apr}}(T,\omega)\leq \lvert \omega \rvert$.
For a (not necessarily $T$-invariant) closed linear subspace $F\subseteq H$ we define the \emph{approximation entropy of $F$ with respect to $T$} by
\begin{align}
    \label{eq:approximation-entropy}
    h_{\mathrm{apr}}(T,F) & \coloneqq \sup_{\substack{\omega \subseteq F \\ \text{finite}}}\, h_{\mathrm{apr}}(T, \omega).
\end{align}
The \emph{approximation entropy of $T$} is given by $h_{\mathrm{apr}}(T) \coloneqq h_{\mathrm{apr}}(T,H)$.
A subset $\omega\subseteq H$ is called a \emph{$T$-generator}, whenever $H$ is the closure of the subspace spanned by $\bigcup_{k\in \mathbb{Z}} T^k\omega$.
For more information on the approximation entropy we refer to~\cite[Section~7]{Voiculescu-DynamicalApproximationEntropies-1995} and the references therein. For reference of the following see~\cite[Proposition 7.4]{Voiculescu-DynamicalApproximationEntropies-1995}.

\begin{lemma}
    \label{lem:PropFourOneVoicu}
    If $(\omega_n)_{n \in \bbN}$ is an increasing sequence of finite subsets of $H$ such that $\bigcup_{n \in \bbN} \omega_n$ is a $T$-generator, then
    $h_{\mathrm{apr}}(T)= \sup_{n \in \bbN} h_{\mathrm{apr}}(T, \omega_n)$.
\end{lemma}

As an immediate consequence of Lemma~\ref{lem:PropFourOneVoicu} we observe the following.

\begin{proposition}
    \label{prop:generator-entropy}
    For any finite-dimensional subspace $F \subseteq H$ and any finite subset $\omega\subseteq F$ with $F = \linSpan(\omega)$ we have
    \begin{align*}
        h_{\mathrm{apr}}(T, F) = h_{\mathrm{apr}}(T, \omega) = h_{\mathrm{apr}}(T\restricted{G})
        = h_{\mathrm{apr}}(T,G),
    \end{align*}
    where $G \coloneqq \overline{\linSpan} \big(\bigcup_{k\in \mathbb{Z}}T^kF \big)$.
    Note that in this context $\omega$ is a $T\restricted{G}$-generator. Moreover,
    \begin{align*}
        h_{\mathrm{apr}}(T) = \sup_F h_{\mathrm{apr}}(T, F),
    \end{align*}
    where the supremum is taken over all finite-dimensional subspaces $F$ of $H$.
\end{proposition}

The following upper bound for the approximation entropy on finite-dimensional subspaces follows easily.

\begin{corollary}
    \label{cor:upper-bound-apr-entropy}
    For finite-dimensional subspaces $F \subseteq H$ we have
    \begin{align*}
        h_{\mathrm{apr}}(T,F) \leq \dim(F).
    \end{align*}
\end{corollary}

For further properties of the approximation entropy see~\cite[Section~7]{Voiculescu-DynamicalApproximationEntropies-1995}.
Recall that for a finite-dimensional subspace  $F \subseteq H$ we call
\begin{align*}
    F_n \coloneqq \linSpan \bigcup_{k = 0}^n T^k F
\end{align*}
the \emph{$n$-th delay subspace} of $F$. Note that the dimension of $F_n$ is upper bounded by $n \cdot \dim(F)$. Furthermore, note that from Proposition~\ref{prop:generator-entropy} it follows that $h_{\mathrm{apr}}(T, F) = h_{\mathrm{apr}}(T, F_n)$.

\begin{theorem}
    \label{thm:linear_lower_bound_dimension}
    Let $H$ be a Hilbert space and $T \colon H \to H$ be unitary, let $F\subseteq H$ be a finite-dimensional linear subspace and $\varepsilon > 0$. Then there exists $K_0 \in \bbN$ and $\delta > 0$ such that whenever for some $n \in \bbN$, for a linear operator $\hat T_n \colon F_n \to F_n$ and $K \geq K_0$ the assertion
    \begin{align*}
        \lVert \hat T^k_n x - T^k x \rVert \leq \delta \lVert x \rVert
    \end{align*}
    holds for all $x \in F_n$ and all $k \in \{0, \dots, K - 1\}$, we conclude that
    \begin{align*}
        \dim(F_n) \geq K (h_{\mathrm{apr}}(T, F) - \varepsilon).
    \end{align*}
\end{theorem}
\begin{proof}
    Let $\omega$ be an orthonormal basis of $F$.
    From Proposition~\ref{prop:generator-entropy} we observe that $h_{\mathrm{apr}}(T, F) = h_{\mathrm{apr}}(T, \omega)$.
    In particular, there exists $\delta>0$, such that $h_{\mathrm{apr}}(T, F) -\frac{\varepsilon}{2}
        \leq h(T,\omega, \delta). $
    Furthermore, there exists $K_0\in \mathbb{N}$, such that for all $K\geq K_0$ we have
    \begin{align*}
        h_{\mathrm{apr}}(T, F) -\varepsilon
        \leq h(T,\omega, \delta)-\frac{\varepsilon}{2}
        \leq \frac{1}{K} H_{\mathrm{apr}} \left( \bigcup_{k = 0}^{K - 1} T^k \omega, \delta \right).
    \end{align*}
    Now consider $K \geq K_0$ and $n \in \mathbb{N}$, such that the assertion
    \begin{align*}
        \lVert \hat T_n^k x - T^k x \rVert \leq \delta \lVert x \rVert
    \end{align*}
    holds for all $x \in F_n$ and all $k \in \{0, \dots, K-1\}$. Denote by $P_{F_n}$ the orthogonal projection onto the $n$-th delay subspace $F_n$.
    From Lemma~\ref{lem:control} we observe that
    \begin{align*}
        \lVert P_{F_n} T^k x - T^k x \rVert \leq \delta \lVert x \rVert = \delta
    \end{align*}
    for all $x \in \omega$ and all $k \in \{0, \dots, K-1\}$.
    In particular, it follows that
    \begin{align*}
        H_{\mathrm{apr}} \left( \bigcup_{k = 0}^{K - 1} T^k \omega, \delta \right) \leq \dim(F_n).
    \end{align*}
    We observe
    \begin{align*}
        h_{\mathrm{apr}}(T, F) - \varepsilon
        \leq \frac{1}{K} H_{\mathrm{apr}} \left( \bigcup_{k = 0}^{K - 1} T^k \omega, \delta \right) \leq \frac{1}{K} \dim(F_n),
    \end{align*}
    and the statement follows.
\end{proof}

\begin{proof}[Proof of Theorem~\ref{thm:approximation-entropy-Koopman}:]
    Since $F=F_0$ the statement follows directly from Theorem~\ref{thm:linear_lower_bound_dimension} considering $n=0$.
\end{proof}

\section{Interplay of the measure-theoretical and the approximation entropy}
\label{sec:interplayMeasureTheoreticalAndApproximationEntropy}

In this section we will show that any measure-preserving dynamical system with positive measure-theoretic entropy has infinite approximation entropy. For this we will explore the following property.

\begin{definition}
    \label{def:coutable-Lebesgue-spectrum}
    Let $T \colon H \to H$ be a unitary operator on a separable Hilbert space $H$. Then $T$ is said to have \emph{countable Lebesgue spectrum} if and only if the there exists an orthonormal sequence ${(x_n)}_{n \in \bbN}$ of $H$ such that for each $n \in \bbN$ the sequence ${(T^k x_n)}_{k \in \bbZ}$ is an orthonormal basis of the \emph{cyclic subspace}
    $Z(x_n; T) \coloneqq \overline{\linSpan}\{ T^k x_n : k \in \bbZ \}$
    and
    \begin{align*}
        H = \bigoplus_{n = 1}^\infty Z(x_n; T).
    \end{align*}
\end{definition}

\begin{remark}
    The reader familiar with spectral theory might be used to define countable Lebesgue spectrum of a unitary operator $T \colon H \to H$ in the context of separable Hilbert spaces by the equivalence of the Lebesgue measure on $\mathbb{T}$ to the spectral measures appearing in the Hahn–Hellinger theorem. This definition is equivalent to our definition as shown in \cite[Proposition 18.32]{EisnerFarkasHaaseNagel-OperatorTheoreticAspectsErgodicTheory-2015}.
\end{remark}

Whenever $(\mathrm{X},\varphi)$ is a measure-preserving dynamical system, then the Koopman operator $T_\varphi$ fixes the constant functions.
Thus no Koopman operator can have countable Lebesgue spectrum without restricting it to certain non-trivial subspaces.
It is thus common to define that $(\mathrm{X}, \varphi)$ has \emph{countable Lebesgue spectrum} if the Koopman operator restricted to $\linSpan \{ \one \}^\bot$ has countable Lebesgue spectrum~\cite[Section~18.4.2]{EisnerFarkasHaaseNagel-OperatorTheoreticAspectsErgodicTheory-2015}. More generally we can  study the following notion, where we purposefully omit the separability assumption on $H$. Note that for operators on separable Hilbert spaces, this definition is equivalent to the notion of countable Lebesgue components given in~\cite[Section~3.7.j]{HasselblattKatok-PrincipalStructures-2002}.

\begin{definition}
    A unitary operator $T \colon H \to H$ on a Hilbert space $H$ is said to have \emph{a countable Lebesgue component} if there exists a closed separable subspace $H_0 \subseteq H$ with $T H_0 = H_0$ such that the $T$ restricted to $H_0$ has countable Lebesgue spectrum.
\end{definition}

We next show that unitary operators with countable Lebesgue components have infinite approximation entropy.

\begin{proposition}
    \label{prop:contLebesgucomp-implies-approxEntInfty}
    Let $T \colon H \to H$ be a unitary operator on a Hilbert space $H$ that has a countable Lebesgue component. Then $h_\mathrm{apr}(T) = \infty$.
\end{proposition}

For its proof we need the following lemma from~\cite[Lemma 7.8]{Voiculescu-DynamicalApproximationEntropies-1995}.

\begin{lemma}
    \label{lem:voiculescuON}
    If $\omega\subseteq H$ is a finite orthonormal system and $\delta > 0$, then
    \begin{align*}
        H_{\mathrm{apr}}(\omega,\delta) \geq \lvert \omega \rvert (1 - \delta^2).
    \end{align*}
\end{lemma}

\begin{proof}[Proof of Proposition~\ref{prop:contLebesgucomp-implies-approxEntInfty}]
    For closed separable subspaces $H_0 \subseteq H$ with $T H_0 = H_0$ it is straight forward to observe that $h_{\mathrm{apr}}(T\restricted{H_0}) \leq h_{\mathrm{apr}}(T)$. We can thus assume without loss of generality that $H$ is separable and that $T$ has countable Lebesgue spectrum.

    Let $(x_n)_{n \in \bbN}$ be the orthonormal sequence from Definition~\ref{def:coutable-Lebesgue-spectrum} and set $\omega_j \coloneqq \{ x_1, \dots, x_j \}$ for each $j \in \bbN$. Then for each $n \in \bbN$ the set $\bigcup_{k = 0}^{n - 1} T^k \omega$ is an orthonormal system with cardinality $j \cdot n$ and from Lemma~\ref{lem:voiculescuON} we obtain
    \begin{align*}
        H_{\mathrm{apr}} \left( \bigcup_{k = 0}^{n - 1} T^k \omega, \delta \right) \geq (j \cdot n) (1 - \delta^2)
    \end{align*}
    for every $\delta > 0$. Then $h_{\mathrm{apr}}(T, \omega_j)  \geq j$ holds for all $j \in \bbN$. Taking the supremum over all $\omega_j$, we obtain $h_{\mathrm{apr}}(T) = \infty$.
\end{proof}

Following~\cite[Section~2.1]{Glasner-ErgodicTheoryViaJoinings-2003} we introduce standard Lebesgue spaces. A measurable space $(X, \Sigma)$, where $X$ is a completely metrizable and second countable topological space and $\Sigma = \calB(X)$ is the Borel-$\sigma$-algebra of $X$, is called \emph{standard Borel}. A measure space $(X, \Sigma, \mu)$ is called \emph{standard Lebesgue} if $(X, \Sigma)$ is a standard Borel and additionally $\mu$ is a regular probability measure.
If $K$ is compact and metrizable, then it is easily checked that $(K, \calB(K), \mu)$ is standard Lebesgue for each probability measure $\mu \colon \calB(K) \to [0,1]$.

From~\cite[Result~14.3]{Rokhlin-EntropyTheoryMeasurePreservingTransformations-1967} and also~\cite[Theorem~3.7.13]{HasselblattKatok-PrincipalStructures-2002} we know the following proposition in the case of $\mathrm{X}$ being a standard Lebesgue space. We will present below how this follows for general invertible measure-preserving systems.

\begin{proposition}
    \label{prop:posMeasureEntropyImpliesCountableLebesgueComponent}
    Let $(\mathrm{X}, \varphi)$ be an invertible measure-preserving dynamical system with $h_\mu(\varphi) > 0$. Then $T_\varphi$ has a countable Lebesgue component.
\end{proposition}

In order to deduce the general statement from the standard Lebesgue space case, we will need the notion of a factor. Let $(\mathrm{X}, \varphi)$ and $(\mathrm{Y}, \psi)$ be measure-preserving dynamical systems on the probability spaces $\mathrm{X} = (X, \Sigma, \mu)$ and $\mathrm{Y} = (Y, \Sigma', \nu)$. Then $\psi$ is called a \emph{factor of $\varphi$}, if there exists an injective Markov operator (also called \emph{Markov embedding}) $T \colon \mathrm{L}^2(\mathrm{Y}) \to \mathrm{L}^2(\mathrm{X})$ that intertwines as $T_\varphi T = T T_\psi$. We refer to Section~\ref{subsec:conditionalExpectations} for a definition of Markov operators.
By~\cite[Theorems~2.3 and~2.4]{Walters-IntroductionToErgodicTheory-1982} this is equivalent to the definition of a factor in~\cite[Section~2.2 on p.~55]{Glasner-ErgodicTheoryViaJoinings-2003}. Note that in the above case the system $\psi$ is sometimes called a \emph{semi-conjugate image} of $\varphi$~\cite[Definition~2.7]{Walters-IntroductionToErgodicTheory-1982}. If $\mathrm{X}$ and $\mathrm{Y}$ are both standard Lebesgue spaces, then the above definition of a factor coincides with the definition of a factor as provided in~\cite[Definition~2.7]{Walters-IntroductionToErgodicTheory-1982}. We refer to the remark succeeding~\cite[Definition~2.7]{Walters-IntroductionToErgodicTheory-1982} for details.

\begin{lemma}
    \label{lem:LebesgueFactors}
    Let $(\mathrm{X}, \varphi)$ be an invertible measure-preserving dynamical system. Then $h_\mu(\varphi) = \sup_\psi h_\nu(\psi)$, where the supremum is taken over all standard Lebesgue factors $(\mathrm{Y}, \psi)$ of $\varphi$.
\end{lemma}
\begin{proof}
    ``$\leq$'': This is a straightforward consequence of the monotonicity of measure-theoretic entropy under factors. For details see~\cite[Section~4.4 on p.~89]{Walters-IntroductionToErgodicTheory-1982}.

    ``$\geq$'':
    Let $\alpha$ be a finite $\Sigma$-measurable partition of $X$ and let $\Sigma'$ be the $\sigma$-algebra generated by the countable set $\bigcup_{n \in \bbN} \bigvee_{k = -n}^n \varphi^k(\alpha)$. Then it follows from a monotone class argument that
    \begin{align*}
        \mathrm{L}^2(X, \Sigma', \nu) = \overline{\linSpan} \bigcup_{n \in \bbN} \left\{\one_A : A \in \bigvee_{k = - n}^n \varphi^k (\alpha) \right\},
    \end{align*}
    where $\nu$ denotes the restriction of $\mu$ to $\Sigma'$ and $\psi \coloneq \varphi$. Denote further $\mathrm{Y} = (X, \Sigma', \nu)$. It follows easily that $\mathrm{L}^2(\mathrm{Y})$ is separable.

    By~\cite[Theorem~12.22]{EisnerFarkasHaaseNagel-OperatorTheoreticAspectsErgodicTheory-2015} there exists measure space $\mathrm{Z} = (K, \calB(K), \xi)$, where $K$ is compact metric and $\xi$ is a probability measure on $\calB(K)$, a measure-preserving dynamical system $(\mathrm{Z}, \kappa)$, and a Markov isomorphism $T_1 \colon \mathrm{L}^2(\mathrm{Z}) \to \mathrm{L}^2(\mathrm{Y})$ that intertwines as $T_\psi T_1 = T_1 T_\kappa$.
    Moreover, the identity operator $T_2  \colon \mathrm{L}^2(\mathrm{Y}) \to \mathrm{L}^2(\mathrm{X})$ is a Markov embedding that trivially intertwines as $T_\varphi T_2 = T_2 T_\psi$. In summary, the operator $T \coloneqq T_2 T_1 \coloneqq \mathrm{L}^2(\mathrm{Z}) \to \mathrm{L}^2(\mathrm{X})$ is a Markov embedding that intertwines as $T_\varphi T = T T_\kappa$. This shows that $\kappa$ is a standard Lebesgue factor of $\varphi$.

    Furthermore, since $(\mathrm{Z}, \kappa)$ and $(\mathrm{Y}, \psi)$ are factors of each other via $T_1$ and $T_1^{-1}$ it follows from monotonicity that $h_\xi(\kappa) = h_\nu(\psi)$. Moreover, the measure-theoretic entropy satisfies $h_\mu(\varphi, \alpha) = h_\kappa(\psi, \alpha) = h_\nu(\psi)$ by the Kolmogorov--Sinai generator theorem~\cite[Theorem~4.17]{Walters-IntroductionToErgodicTheory-1982}. This proves the claim.
\end{proof}

\begin{lemma}
    \label{lem:extensionCountableLebesgueComponent}
    Consider a factor $(\mathrm{Y}, \psi)$ of a measure-preserving dynamical system $(\mathrm{X}, \varphi)$. Whenever $T_\psi$ has a countable Lebesgue component, then so does $T_\varphi$.
\end{lemma}
\begin{proof}
    Denote by $T \colon \mathrm{L}^2(\mathrm{Y}) \to \mathrm{L}^2(\mathrm{X})$ the associated Markov embedding that intertwines $T_\varphi$ and $T_\psi$ as $T_\varphi T = T T_\psi$
    and $H_0 \subseteq \mathrm{L}^2(\mathrm{Y})$ for a closed separable subspace with $T_\psi H_0 = H_0$, such that $T_\psi$ restricted to $H_0$ has countable Lebesgue spectrum. 
    Clearly $H_1 \coloneqq T H_0$ is a closed separable subspace of $\mathrm{L}^2(\mathrm{X})$ with $T_\varphi H_1 = H_1$. Since any Markov embedding is isometric by~\cite[Theorem 13.9]{EisnerFarkasHaaseNagel-OperatorTheoreticAspectsErgodicTheory-2015} it preserves orthonormality. Hence, $T_\varphi$ restricted to $H_1$ has a countable Lebesgue spectrum.
\end{proof}

\begin{proof}[Proof of Proposition~\ref{prop:posMeasureEntropyImpliesCountableLebesgueComponent}:]
    By Lemma~\ref{lem:LebesgueFactors} there exists a standard Lebesgue factor $(\mathrm{Y}, \psi)$ of $\varphi$ with $h_\nu(\psi) > 0$.
    Since we know the statement in the context of standard Lebesgue systems~\cite[Result~14.3]{Rokhlin-EntropyTheoryMeasurePreservingTransformations-1967} the Koopman operator $T_\psi$ has a countable Lebesgue component. As $\psi$ is a factor of $\varphi$ we observe that also $T_\varphi$ has a countable Lebesgue component from Lemma~\ref{lem:extensionCountableLebesgueComponent}.
\end{proof}

The following theorem now follows from Propositions~\ref{prop:contLebesgucomp-implies-approxEntInfty} and~\ref{prop:posMeasureEntropyImpliesCountableLebesgueComponent}.

\begin{theorem}
    \label{thm:hApproxInfinity}
    Let $(\mathrm{X}, \varphi)$ be an invertible measure-preserving dynamical system with $h_\mu(\varphi) > 0$. Then $h_{\mathrm{apr}}(T_\varphi) = \infty$.
\end{theorem}

There exist invertible measure-preserving dynamical systems $(\mathrm{X}, \varphi)$ with $h_\mu(\varphi) = 0$ such that $T_\varphi$ has a countable Lebesgue component (and hence $h_{\mathrm{apr}} (T_\varphi)= \infty$), as the following example demonstrates.

\begin{example}
    Let $\alpha \in \bbR$ be irrational. Consider the skew-rotation on the two-dimensional torus
    \begin{align*}
        \varphi \colon \bbT^2 \to \bbT^2, \quad (x,y) \mapsto (\mathrm{e}^{\pi \mathrm{i} \alpha} x, x y),
    \end{align*}
    and equip $\mathbb{T}$ with the Lebesgue measure $\mu$.
    The mapping $\varphi$ becomes an invertible measure-preserving system.
    Moreover, it has a countable Lebesgue component~\cite[Proposition~18.36]{EisnerFarkasHaaseNagel-OperatorTheoreticAspectsErgodicTheory-2015}. Hence, $h_{\mathrm{apr}}(T_{\varphi}) = \infty$.

    However, the measure-theoretic entropy $h_{\mu}(\varphi)$ is zero. Indeed, as presented in~\cite[Chapter 7]{AuslanderMinimalFLows-1988} the respective topological dynamical system is distal and we observe from ~\cite[Corollary~6.7]{Yan-ConditionalEntropyFiberEntropyAmenableGroupActions-2015} that it has zero topological entropy. From the variational principle~\cite[Theorem 8.6]{Walters-IntroductionToErgodicTheory-1982} it follows that $h_{\mu}(\varphi) = 0$.
\end{example}

\section{A spectral theoretic point of view}
\label{sec:theSpectralPerspective}

In this final section we present some details of known results on the relationship between the approximation entropy and the spectral theory of unitary operators~\cite{Conway-CourseFunctionalAnalysis-1985, EisnerFarkasHaaseNagel-OperatorTheoreticAspectsErgodicTheory-2015,Kriete-AnElementaryApproachMultiplicityTheory-1986,Voiculescu-DynamicalApproximationEntropies-1995}.

Let $T \colon H \to H$ be a unitary operator on a complex and separable Hilbert space and assume that $\bbK = \bbC$. For $x \in H$ we call the unique measure $\mu_x \colon \calB(\bbT) \to \bbR$ satisfying the property
\begin{align*}
    \langle x, T^k x \rangle = \int_\bbT z^{-k} \, \mu_x(\mathrm{d} z)
\end{align*}
for all $k \in \bbZ$ a \emph{spectral measure} of $T$. That such a measure exists follows from the Bochner--Herglotz theorem~\cite[Theorem~18.6]{EisnerFarkasHaaseNagel-OperatorTheoreticAspectsErgodicTheory-2015}. A spectral measure $\mu_x$ can be decomposed by means of a Lebesgue decomposition (see~\cite[p.~22-25]{ReedSimon-FunctionalAnalysisI-1980} or~\cite[Theorem~B.25]{EisnerFarkasHaaseNagel-OperatorTheoreticAspectsErgodicTheory-2015}) into mutually singular measures
\begin{align*}
    \mu_x = \mu_{x, \mathrm{d}} + \mu_{x, \mathrm{ac}} + \mu_{x, \mathrm{sc}},
\end{align*}
where $\mu_{x, \mathrm{d}}$ is a discrete measure, $\mu_{x, \mathrm{ac}}$ is absolutely continuous with respect to the Lebesgue measure $\mathfrak{m}$ on $\bbT$ (in short $\mu_{x ,\mathrm{ac}} \ll \mathfrak{m}$), and $\mu_{x, \mathrm{sc}}$ is non-atomic and singular to the Lebesgue measure (in short $\mu_{x ,\mathrm{sc}} \bot \mathfrak{m}$). Note that all three parts are absolutely continuous with respect to $\mu_x$, and therefore, there exists elements $x_\mathrm{d}, x_\mathrm{ac}, x_\mathrm{sc} \in H$ such that $\mu_{x, \mathrm{*}} = \mu_{x_\mathrm{*}}$ for $* \in \{\mathrm{d}, \mathrm{ac}, \mathrm{sc}\}$~\cite[Lemma~18.9]{EisnerFarkasHaaseNagel-OperatorTheoreticAspectsErgodicTheory-2015}. This yields a decomposition of the Hilbert space $H = H_\mathrm{d} \oplus H_\mathrm{ac} \oplus H_\mathrm{sc}$ into $T$-invariant subspaces
\begin{align*}
    H_\mathrm{d}    & \coloneqq \{ x \in H : \mu_x \text{ is discrete} \}, \quad
    H_\mathrm{ac} \coloneqq \{x \in H : \mu_x \ll \mathfrak{m} \},                                        \\
    H_{\mathrm{sc}} & \coloneqq \{ x \in H : \mu_x \text{ is non-atomic and } \mu_x \bot \mathfrak{m} \}.
\end{align*}
The $T$-invariance follows readily from the fact that $\mu_{Tx} = \mu_x$ for all $x \in H$, as $T$ is unitary. Furthermore, the decomposition into $T$-invariant subspaces yields a decomposition of $T$ into
\begin{align*}
    T = T_\mathrm{d} \oplus T_\mathrm{ac} \oplus T_\mathrm{sc}
\end{align*}
with $T_* \colon H_* \to H_*$ for all $* \in \{\mathrm{d}, \mathrm{ac}, \mathrm{sc}\}$.

\begin{remark}
    \label{rem:KroneckerFactor}
    \begin{enumerate}[label = (\roman*)]
        \item
              The subspace  $H_{\mathrm{d}}$ is given by the closure of the linear span of all eigenvectors of $T$~\cite[Proposition~18.17]{EisnerFarkasHaaseNagel-OperatorTheoreticAspectsErgodicTheory-2015}.

        \item
              If $T = T_{\mathrm{rev}} \oplus T_{\mathrm{aws}}$ is a Jacob--de~Leeuw--Glicksberg decomposition~\cite[Chapter~16.4]{EisnerFarkasHaaseNagel-OperatorTheoreticAspectsErgodicTheory-2015} into a \emph{reversible part} $T_\mathrm{rev}$ and an \emph{almost weakly stable part} $T_\mathrm{aws}$ decomposition of $T$, then $T_{\mathrm{d}}$ and $T_{\mathrm{rev}}$ coincide~\cite[Corollary~18.18]{EisnerFarkasHaaseNagel-OperatorTheoreticAspectsErgodicTheory-2015}.

        \item
              The operator $T_{\mathrm{c}} \coloneqq T_{\mathrm{ac}} \oplus T_{\mathrm{sc}}$ on $H_{\mathrm{c}} \coloneqq H_{\mathrm{ac}} \oplus H_{\mathrm{sc}}$ coincides with the almost weakly stable part $T_{\mathrm{aws}}$ of the Jacob--de~Leeuw--Glicksberg decomposition~\cite[Corollary~18.18]{EisnerFarkasHaaseNagel-OperatorTheoreticAspectsErgodicTheory-2015}. The elements $x \in H_{\mathrm{c}}$ have the property that
              \begin{align*}
                  \lim_{n \to \infty} \frac{1}{n} \sum_{k = 0}^{n - 1} \lvert \langle T^k x, y \rangle \rvert^p = 0
              \end{align*}
              holds for all $y \in H$ and all $p \in [1, \infty)$~\cite[Theorem~16.34]{EisnerFarkasHaaseNagel-OperatorTheoreticAspectsErgodicTheory-2015}. It is also often referred to as the \emph{continuous part} of $T$~\cite[Chapter~18.2]{EisnerFarkasHaaseNagel-OperatorTheoreticAspectsErgodicTheory-2015}.
    \end{enumerate}
\end{remark}

Next we define the concept of the multiplicity function by following~\cite[Section~2]{Kriete-AnElementaryApproachMultiplicityTheory-1986}. Let $\ell^2 \coloneqq \ell^2(\bbN; \bbC)$ denote the Hilbert space of complex square summable sequence indexed by the natural numbers. Further let $\nu \colon \calB(\bbT) \to [0, \infty)$ be a Borel measure on $\bbT$, write $\mathrm{N} \coloneqq (\bbT, \calB(\bbT), \nu)$ for the corresponding measure space and denote by $\mathrm{L}^2(\mathrm{N}; \ell^2)$ the space of all functions $f \colon \bbT \to \ell^2$ such that $\int_\bbT \lVert f \rVert_{\ell^2}^2 \, \mathrm{d} \nu < \infty$. The scalar product $\langle f, g \rangle  = \int_\bbT \langle f(\lambda), g(\lambda)\rangle_{\ell^2} \, \nu(\mathrm{d}\lambda)$ on $\mathrm{L}^2(\mathrm{N}; \ell^2)$ makes it into a Hilbert space.

Now let $m \colon \bbT \to \bbN_0 \cup \{\infty\}$ be Borel measurable function with $\nu(\{m = 0 \}) = 0$ and set $\ell^2_n \coloneqq \{ x = (x_n)_{n \in \bbN} \in \ell^2 : \forall k > n : x_k = 0 \}$ for all $n \in \bbN_0 \cup \{\infty\}$. We denote
\begin{align*}
    H_{\nu, m} \coloneqq \{ f \in \mathrm{L}^2(\bbN; \ell^2) : f(\lambda) \in \ell^2_{m(\lambda)}  \text{ for $\nu$-almost all $\lambda \in \bbT$} \}.
\end{align*}
Then $H_{\nu, m}$ is a closed subspace of $\mathrm{L}^2(\bbN; \ell^2)$. Consider the multiplication operator
\begin{align*}
    T_{\nu, m} \colon H_{\nu, m} \to H_{\nu, m}, \quad (T_{\nu, m} f)(\lambda) \mapsto \lambda \cdot f(\lambda)
\end{align*}
by the identity function.

The representation theorem for unitary operator (or more generally normal operators) states that every unitary operator $T \colon H \to H$ on a separable Hilbert space is unitarily equivalent to $T_{\nu, m}$ for some Borel measure $\nu$ and some Borel function $m$ as above (see~\cite[Theorem~II.9.18]{Conway-SubnormalOperators-1981} or~\cite[Theorem~IX.10.20]{Conway-CourseFunctionalAnalysis-1985}). Moreover, if $\tilde \nu$ is another Borel measure and $\tilde m$ another Borel function with this property, then $\nu$ and $\tilde \nu$ have the same null sets (or equivalently, are mutually absolutely continuous) and $\nu$-almost everywhere we have $m = \tilde m$~\cite[Corollary~II.9.12~b)]{Conway-SubnormalOperators-1981}.
In this case, the measure $\nu$ is called a \emph{maximal spectral measure} and the function $m$ is called a \emph{multiplicity function} of $T$.

The maximal spectral measure $\nu$ is indeed a spectral measure of $T$. To see this denote by $x$ the constant function in $\mathrm{L}^2(\mathrm{N}; \ell^2)$ that maps to the first unit vector $e_1 \in \ell^2$. Since $\nu(\{ m = 0 \}) = 0$, it follows that $x \in H_{\nu, m}$. Then
\begin{align*}
    \langle x, T_{\nu, m}^k x \rangle_{H_{\nu, m}} & = \int_{\bbT} \langle x, T_{\nu, m} x \rangle_{\ell^2} \, \mathrm{d} \nu = \int_{\bbT} z^{-k} \, \nu(\mathrm{d} z)
\end{align*}
holds for all $k \in \bbZ$, from which follows that $\nu = \mu_x$. Thus, by unitary equivalence of $T$ and $T_{\nu, m}$ the statement follows.

Suppose now that $T = T_\mathrm{ac}$ and let $\nu$ be a maximal spectral measure and $m$ be a multiplicity function of $T$. Then there exists $x \in H = H_\mathrm{ac}$ such that $\nu  = \mu_x$. In particular, $\nu$ is absolutely continuous with respect to the Lebesgue measure $\mathfrak{m}$ on $\bbT$. Let $\frac{\mathrm{d} \nu}{\mathrm{d} \mathfrak{m}} \in \mathrm{L}^1(\bbT, \calB(\bbT), \mathfrak{m})$ be its Radon--Nikodým derivative. Denote by $\Delta_T \coloneqq \{\frac{\mathrm{d} \nu}{\mathrm{d} \mathfrak{m}} \neq 0\}$ its set of non-vanishing points, which up to a $\mathfrak{m}$-null set is uniquely defined and does not depend on the maximal spectral measure $\nu$ chosen. Then $\nu$ and $\mathfrak{m}\restricted{\Delta_T}$ are mutually absolutely continuous and can therefore be exchanged in the representation of $T$.

The following was proved for the case in which the multiplicity function for the absolutely continuous part $T_\mathrm{ac}$ is finite in~\cite[Theorem~7.11]{Voiculescu-DynamicalApproximationEntropies-1995}. We present the short argument how the case of a general multiplicity function follows from this.

\begin{proposition}
    \label{prop:apr-entropy-absolutely-cont-spetrum}
    Let $T \colon H \to H$ be a unitary operator on a complex and separable Hilbert space and denote by $m$ the multiplicity function of its absolutely continuous part $T_\mathrm{ac}$. Then
    \begin{align*}
        h_{\mathrm{apr}}(T) = \int_{\Delta_{T_\mathrm{ac}}} m \, \mathrm{d} \mathfrak{m}.
    \end{align*}
    In particular, if $H_\mathrm{ac} = \{ 0 \}$, then $h_\mathrm{apr}(T) = 0$.
\end{proposition}

For the integral on the right-hand side in Proposition~\ref{prop:apr-entropy-absolutely-cont-spetrum} to be well-defined, the restriction $\Delta_{T_\mathrm{ac}}$ is necessary, as the following example shows.

\begin{example}
    \label{ex:waring-voiculescu}
    Let $\tilde{\mathfrak{m}}$ be the Lebesgue measure supported on $\bbT_ + \coloneqq\{\lambda \in \bbT : \operatorname{Re}(\lambda) \geq 0 \}$. The multiplication operator $T \colon \mathrm{L}^2( \bbT, \calB(\bbT), \tilde{\mathfrak{m}}) \to \mathrm{L}^2( \bbT, \calB(\bbT), \tilde{\mathfrak{m}})$ defined by $f \mapsto \id \cdot f$ has maximal spectral measure $\tilde{\mathfrak{m}}$, and thus, both functions $m_1 \coloneqq \one_{\bbT}$ and $m_2 \coloneqq \one_{\bbT_+}$ are multiplicity functions for $T$.
\end{example}

\begin{proof}[Proof of Proposition~\ref{prop:apr-entropy-absolutely-cont-spetrum}]
    As shown in the proof of~\cite[Theorem~7.11]{Voiculescu-DynamicalApproximationEntropies-1995} we may assume that $T = T_\mathrm{ac}$. Let $m$ be a multiplicity function of $T$.

    If $m < \infty$, then the statement of Proposition~\ref{prop:apr-entropy-absolutely-cont-spetrum} is proved in~\cite[Theorem~7.11]{Voiculescu-DynamicalApproximationEntropies-1995} by showing that on the additive semigroup $\calC$ of functions $m \colon \bbT \to \bbN_0$ the functional
    \begin{align*}
        \mu \colon \calC \to [0, \infty), \quad m \mapsto h_\mathrm{apr}(T),
    \end{align*}
    where $T = T_{\nu, m}$ and $\nu = \mathfrak{m}\restricted{\{m \neq 0 \}}$, coincides with the integral $\int_\bbT m \, \mathrm{d}\mathfrak{m} = \int_{\Delta_{T}} m \, \mathrm{d} \mathfrak{m}$.

    Now consider the general case and let $T_{\nu, m}$ be a representation of $T$ with $\nu = \mathfrak{m}\restricted{\{m \neq 0 \}}$. Consider the closed subspaces
    \begin{align*}
        H_n \coloneqq \{ f \in H_{\nu, m} : f \in \ell^2_n\}
    \end{align*}
    for each $n \in \bbN$.
    Then $H_n \subseteq H_{n + 1}$ and $\bigcup_{n \in \bbN} H_n = H_{\nu, m}$. Moreover, $T_n \coloneqq T\restricted{H_n}$ has the multiplicity function $m_n \coloneqq \min(m, n)$ with $\nu$ being a maximal spectral measure of $T_n$. Thus, by the first part of the proof
    \begin{align*}
        h_\mathrm{apr}(T, H_n) = h_\mathrm{apr}(T_n) = \int_{\Delta_{T}} m_n \, \mathrm{d} \mathfrak{m}.
    \end{align*}
    Now increasing $n \to \infty$, we obtain on the right-hand side of the equation from the monotone convergence theorem that $\sup_{n \in \bbN} \int_{\Delta_{T}} m_n \, \mathrm{d} \mathfrak{m} = \int_{\Delta_T} m \, \mathrm{d} \mathfrak{m}$. On the left-hand side we have $h_\mathrm{apr}(T, H_n) = h_\mathrm{apr}(T_n)$ and obtain from Lemma~\ref{lem:PropFourOneVoicu} together with Proposition~\ref{prop:generator-entropy} that $\sup_{n \in \bbN} h_\mathrm{apr}(T_n) = h_\mathrm{apr}(T_{\nu, m}) = h_\mathrm{apr}(T)$.
\end{proof}

For different constructions of a representation of normal operators using Hahn--Hellinger theory we refer to~\cite[Theorems~IX.10.20 and~IX.10.21]{Conway-CourseFunctionalAnalysis-1985},~\cite[Theorems~1.34 and~1.36]{Nadkarni-SpectralTheoryDynamicalSystems-2020} or~\cite[Theorems~II.4 and~II.11]{Queffelec-SubstitutionDynamicalSystemsSpectralAnalysis-2010}.
Spectral theory of dynamical systems is surveyed in~\cite{KatokThouvenot-SpectralPropertiesCombinatorialConstructions-2006}; more specifically to Section~2.1 therein for more information on dynamical systems with countable Lebesgue part. The latter is still an active field of research~\cite{Lemanczyk-OnSomeSpectralProblemsInErgodicTheory-2024} and has been for some time~\cite[Part~III]{CornfeldFominSinai-ErgodicTheory-1982}.

\subsection*{Acknowledgements}
The first named author is funded by the Deutsche Forschungsgemeinschaft (DFG, German Research Foundation) – 530703788.
The authors thank Henrik Kreidler for pointing us to the reference~\cite{Voiculescu-DynamicalApproximationEntropies-1995}.

\printbibliography%

\end{document}